\documentclass[12pt,leqno]{amsart}
\usepackage[all]{xy}
\usepackage{amssymb,mathrsfs,euscript,enumitem,marvosym}
\usepackage{mathrsfs,euscript,amssymb,amsmath,bbold}
\usepackage[hypertex]{hyperref}
\textheight9in  \def\DATE{August 10, 2016}
\newtheorem{theorem}{Theorem}

\newtheorem{lemma}[theorem]{Lemma} 
\newtheorem{proposition}[theorem]{Proposition}

\theoremstyle{definition}

\newtheorem{definition}[theorem]{Definition}

\def\rrule{{rule}}
\def\odotM{{\odot_{\rm M}}}
\def\odotS{{\odot_{\rm S}}}
\def\ttA{{\tt A}}\def\ttB{{\tt B}}
\def\CV{{\tt Cat}(\Vect)}
\def\op{operad}
\def\bbA{{\mathbb N}}
\def\dual{*}

\def\susp{\uparrow\!}
\def\desusp{\downarrow\!}

\makeatletter
\newcommand{\downar}{\@ifnextchar*{\starDownAr}{\nostarDownAr}}
\def\starDownAr*{
  {\raisebox{0.5pt}[0pt][0pt]{\ensuremath{\scriptstyle{\downarrow}}}}
} \def\nostarDownAr{
  {\raisebox{1pt}[0pt][0pt]{\ensuremath{\downarrow}}} } 
\makeatother
\makeatletter
\newcommand{\upar}{\@ifnextchar*{\starUpAr}{\nostarUpAr}}
\def\starUpAr*{ {\raisebox{1pt}[0pt][0pt]{\ensuremath{\scriptstyle{\uparrow}}}} }
\def\nostarUpAr{ {\raisebox{1pt}[0pt][0pt]{\ensuremath{\uparrow}}} }
\makeatother

\def\twxxi{\bullet}
\def\vrt{|}

\allowdisplaybreaks

\def\oM{{\EuScript O}}
\def\oM{\hbox{\Leo}}

\def\Vect{{\tt Vect}}
\def\Cor{{\tt fSet}}

\def\stt#1{{\{#1\}}}
\def\Rada#1#2#3{{#1_{#2},\ldots,#1_{#3}}}
\def\rada#1#2{{#1,\ldots,#2}}

\def\ext{{\mathbb S}}

\def\redukce#1{\vbox to .3em{\vss\hbox{#1}}}
 \def\({{\hbox{$(\!($}}}\def\){{\hbox{$)\!)$}}}
\def\End{{\mathcal E} \hskip -.1em{\it nd\/}}

\def\lra{\longrightarrow}

\def\vlra{{\hbox{$-\hskip-1mm-\hskip-2mm\longrightarrow$}}}

\def\bbZ{{\mathbb Z}}\def\Lin{{\underline {\tt Vect}}}

\def\ot{\otimes}\def\bfk{{\mathbb k}}

\newcommand{\twooo}[2]{\sideset{_{#1}}{_{#2}}{\mathop{\bullet}}}
\long\def\martin#1\endmartin{
\todo[inline]{#1}
}
\makeatletter
\providecommand\@dotsep{5}
\def\listtodoname{List of Todos}
\def\listoftodos{\@starttoc{tdo}\listtodoname}
\makeatother

\long\def\comment#1\endcomment{\relax}

\def\id{{\mathbb 1}}
\newcommand{\sid}{\hbox{\scriptsize $1 \hskip -.35em 1$}}

\def\({(} \def\){)} 
\def\bl{\big(} \def\br{\big)}

\title{Odd structures are odd}

\author{Martin Markl}

\catcode`\@=11
\address{Mathematical Institute of the Academy, {\v Z}itn{\'a} 25,
         115 67 Prague 1, The Czech Republic}
\address{Faculty of Mathematics and Physics, Charles University,
186 75 Sokolovsk\'a 83, Prague~8, The Czech Republic}
\email{markl@math.cas.cz}

\def\@evenfoot{\rule{0pt}{20pt}[\DATE] \hfill [{\tt \jobname.tex}]}
\def\@oddfoot{\rule{0pt}{20pt}{[\tt \jobname.tex}]\hfill [\DATE]}

\catcode`\@=13

\thanks{The author was supported by the Eduard \v Cech Institute
  P201/12/G028 and RVO: 67985840.}

\begin{document}
\bibliographystyle{plain}

\begin{abstract}
By an odd structure we mean  an algebraic structure in the
category of graded vector spaces whose structure operations have odd
degrees. Particularly important are odd modular
{\op}s which appear as Feynman transforms of modular
{\op}s and, as such, describe some structures
of string field theory. 

We will explain how odd structures are affected by the choice of the monoidal
structure of the underlying category.   
We will then present two `natural' and `canonical' 
constructions of an odd modular endomorphism {\op} leading to  different
results, only one being correct.
This contradicts the generally accepted belief that the systematic use of the
Koszul sign \rrule\ leads to correct signs. 
\end{abstract}

\keywords{Graded vector space, monoidal structure,
odd endomorphism operad} 
\subjclass[2010]{18D50 (Primary), 18D20, 18D10 (Secondary)}

\maketitle

\tableofcontents

\section*{Introduction}
\renewcommand*{\thefootnote}{\fnsymbol{footnote}}
\setcounter{footnote}{1}

As noticed in the seminal paper~\cite{getzler-kapranov:CompM98}, the
category of modular {\op}s is not Koszul self-dual. Consequently, the bar
construction of a modular {\op} (called in this context the {\em Feynman
transform\/}) is not an ordinary modular {\op}, but an {\em
  odd\/} modular {\op}.\footnote{Terminology suggested by Ralph Kaufmann; a
  definition is recalled in Section~\ref{zitra_letim_do_Prahy}.}
It~was shown
in \cite{markl:la} and in the forthcoming work
\cite{doubek:_modul,muenster,munster-sachs} that some algebraic structures
relevant for string field theory are algebras over the Feynman
transform of a modular {\op}. This explains the interest in explicit
understanding odd modular operads and their algebras.

Odd modular operads live in the category $\Lin$ of graded
vector spaces and their linear homogeneous maps of arbitrary degrees.
This category is enriched over the category $\Vect$ of graded vector
spaces and linear maps of degree $0$ and admits two different yet
`natural' and `canonical' symmetric closed monoidal structures.

Structural 
operations of odd modular operads have degrees $+1$, so they are examples of an
`odd' structure having
operations of odd degrees. We use odd
associative algebras as a~simple example which shows 
that concrete `models' of such structures might
depend on the choice of a monoidal structure of $\Lin$. The same is
very crucially true also for odd 
modular operads, the main subject of this note. 
The `oddness' of operad- and PROP-like structures was discussed in
great detail in \cite{ten-vul-Kaufmann}; we refer to it for
other examples of odd structures.

The concept of algebras over odd modular
{\op}s requires odd endomorphism {\op}s; and algebra over an odd
modular operad $\oM$ is a morphism $\oM \to \End_V$ from $\oM$ to
the odd endomorphism operad $\End_V$. The structure of the odd
endomorphism operad is extremely sign-sensitive.
We will present {\it two} `natural' and `canonical' 
constructions of this {\op}, leading quite 
unexpectedly to {\it different\/} results. This
phenomenon is explained by the presence of {\it two}
different monoidal structures of $\Lin$. 
Proposition~\ref{pokusim_se_vratit_kufr} of the last section
specifies which one gives 
the correct result.

The moral is that even the systematic and careful use of the Koszul
sign \rrule\ might lead to wrong results if one is unlucky.  
We hope that this note would warn the reader that this may indeed happen 
if `odd' structures are present.

\vskip .3cm

\noindent 
{\bf Conventions.}
All algebraic objects will be considered over a fixed field $\bfk$ of
characteristic zero.  The symbol $\otimes$ will be reserved for the
tensor product over $\bfk$.  We will denote by $\id_X$ or simply by
$\id$ when $X$ is understood, the identity endomorphism of an object
$X$ (set, vector space, \&c.).

By a {\em grading\/} we mean a ${\bbZ}$-grading, though everything in
this note can easily be modified to the ${\bbZ}_2$-graded case.
The degree of a graded object will be denoted by $|w|$.  
We will  use the
Koszul sign \rrule\ meaning that whenever we commute two ``things''
of degrees $p$ and $q$, respectively, we multiply the sign by~$(-1)^{pq}$. 

We  assume basic
knowledge of operads with the emphasis on modular ones as it can
be gained for example from~\cite[Chapter
5]{markl-shnider-stasheff:book} 
complemented by the original
source \cite{getzler-kapranov:CompM98}. 

\vskip .15cm

\noindent 
{\bf Notation.}
For $n \geq 1$ we denote by $\Sigma_n$  the symmetric group of $n$
elements realized as the group of automorphism of the
set $\stt{\rada1n}$. 
For graded indeterminates $\Rada x1n$ and a permutation $\sigma\in
\Sigma_n$ we define the {\em Koszul sign\/}
$\epsilon(\sigma)=\epsilon(\sigma;\Rada x1n)$ by
\begin{equation}
\label{Koszul_sign}
x_1 \cdots  x_n = \epsilon(\sigma;x_1,\dots,x_n)
\cdot x_{\sigma(1)} \cdots  x_{\sigma(n)},
\end{equation}
which has to be satisfied in the free graded commutative associative algebra
$\ext(\Rada x1n)$ generated by $\Rada x1n$.

For graded vector spaces $V$ and $W$ we denote by $\Lin^k(V,W)$ the vector
space of degree $k$ morphisms $V \to W$ and  by $\Lin(V,W)$
the graded vector space
\[
\Lin(V,W) := \bigoplus_{k \in {\mathbb Z}} \Lin^k(V,W).
\] 
If $W$ is the ground field
$\bfk$, we obtain the graded dual $V^\dual :  = \Lin(V,\bfk)$ of
$V$.
Notice that the degree $k$ component of $V^\dual$ equals the standard
linear dual $(V^{-k})^\dual$  of the degree $-k$ component of $V$.
A degree $k$ morphism $f : V \to W$ defines a map \hbox{$f^\dual : W^\dual \to
V^\dual$} of the same degree by the formula
\begin{equation}
\label{dostane_Jarka_dovcu?}
f^\dual (x) := (-1)^{k |x|}\, x \circ f, \ x \in V^*.
\end{equation}
For $V=\bigoplus_p V_p$\,, let $\susp V$ 
be the suspension of $V$, i.e.\
the graded
vector space defined by $(\susp V)_p = V_{p-1}$.
One has the obvious linear isomorphisms $\uparrow\ : V \to\ \susp V$ and
$\downarrow\ :\ \susp V\to  V$ of degrees $+1$ and $-1$, respectively.
Given graded vector spaces $V_1$ and $V_2$, we denote by $\tau$ the symmetry
\begin{equation}
\label{flip}
\tau : V_1 \ot V_2 \to V_1 \ot V_2,\ \tau (v_1 \ot v_2) :=  
(-1)^{|v_1||v_2|} (v_2 \ot v_1).
\end{equation}

\vskip .3cm

\noindent 
{\bf Acknowledgment.}
Main ideas of this note were born during
conversations with Michael Batanin and Martin Doubek. Martin 
pointed to me the potential sign problems in the construction of the
odd endomorphism operad, and Michael substantially contributed to my
understanding of the categories of graded vector spaces. 
I owe my thanks to Bruno Vallette for explaining to me the
difference between the Koszul sign convention and the Koszul sign rule, and
to Ralph Kaufmann for turning my attention to \cite{ten-vul-Kaufmann}.
I enjoyed the wonderful
atmosphere of the Max-Planck Institut
f\"ur Matematik in Bonn 
in the period when this paper \hbox{was~completed}. I am also indebted
to an anonymous referee for many useful remarks and corrections.

\vskip .3cm

\noindent 
{\bf Plan of the paper.} In Section~\ref{zase_prsi}
we show that the $2$-category $\CV$ of categories enriched over the category
$\Vect$ of graded vector
spaces and their linear degree $0$-maps admits two monoidal
structures, $\odotS$ and $\odotM$, thus one has two types of pseudomonoids in
$\CV$. We notice that the category $\Lin$  of graded vector
spaces and their linear maps of {\em arbitrary\/} degrees with its
standard $\Vect$-enriched monoidal (tensor) structure is a
pseudomonoid for~$\odotS$, while it has yet another monoidal structure
which makes it a~pseudomonoid for $\odotM$.  Section~\ref{zase_prsi}
is complemented  with
a toy example of an odd structure in~$\Lin$.

In Section~\ref{Mikes_mi_popral_k_narozkam} we recall
odd modular operads and show how their concrete models depend on
the choice of a monoidal structure of $\Lin$. The last section is
devoted to two constructions of the odd endomorphism operad. We
investigate the
sensitivity of these constructions to the monoidal structure of $\Lin$.

\section{The category of graded vector spaces}
\label{zase_prsi}

Let  $\Vect$  denote the category of graded vector
spaces and their linear degree $0$-maps, and $\CV$ the $2$-category of
$\Vect$-enriched categories. The $2$-category $\CV$ bears the 
`standard' monoidal structure, denoted  $\odotS$, defined as
follows. For enriched categories $\ttA,\ttB \in \CV$,  the objects of the
category $\ttA \odotS \ttB$ are  couples $(a,b)$, where $a$ is an object of
$\ttA$ and $b$ an object of $\ttB$. The enriched hom-spaces are
\[
(\ttA \odotS \ttB)\big((a_1,b_1),(a_2,b_2)\big) 
:= \ttA(a_1,a_2) \ot \ttB(b_1,b_2),
\]
and the enriched composition $\circ_{\rm S}$ is given by the diagram
\def\otr{{{\ot}}}
\[
\xymatrix@C = -1em@R = 2em{(\ttA \odotS \ttB)\big((a_1,b_1),(a_2,b_2)\big) \otr (\ttA
  \odotS \ttB)\big((a_2,b_2),(a_3,b_3)\big) \ar[r]^-{\circ_{\rm S}}
\ar@{=}[dd]
&(\ttA \odotS \ttB)\big((a_1,b_1),(a_3,b_3)\big) 
\\
&\ttA(a_1,a_3) \otr \ttB(b_1,b_3) \ar@{=}[u]
\\
\hskip -5em
\ttA(a_1,a_2) \otr \ttB(b_1,b_2) \otr  \ttA(a_2,a_3) \otr \ttB(b_2,b_3) 
\ar[r]^-{(\id \ot \tau \ot \id)}&
\ar[u]^{\circ\, \ot \, \circ} \hphantom{.}
\ttA(a_1,a_2) \otr \ttA(a_2,a_3) \otr  \ttB(b_1,b_2) \otr
\ttB(b_2,b_3).
}
\]
In the above displays, $a_1,a_2, a_3$ are objects of $\ttA$, $b_1,b_2,
b_3$ objects of $\ttB$,  $\tau$ the symmetry~(\ref{flip}) and $\circ$
the compositions in $\ttA$ resp.~$\ttB$. 
The $\circ_S$-composition can also be defined directly by the~formula
\[
(f_1 \ot g_1)\circ_{\rm S} (f_2 \ot g_2) :=
(-1)^{|g_1||f_2|} (f_1\circ f_2 \ot g_1\circ g_2),
\]
where $f_1,f_2$ are composable morphisms of $\ttA$ and 
$g_1,g_2$ composable morphisms of $\ttB$. 

The $2$-category $\CV$ has however another 
monoidal structure which we denote by $\odotM$ and
call it, from the reasons
explained later, the {\em McGill\/} monoidal structure.
The objects of $\ttA \odotM \ttB$ are the same
as the objects of $\ttA \odotS \ttB$, and also the enriched hom-spaces
agree, i.e.\ 
\[
(\ttA \odotM \ttB)\big((a_1,b_1),(a_2,b_2)\big) 
:= \ttA(a_1,a_2) \ot \ttB(b_1,b_2),
\]
but the composition $\circ_{\rm M}$ is now given by 
\[
(f_1 \ot g_1)\circ_{\rm M} (f_2 \ot g_2) :=
(-1)^{|f_1||g_2|} (f_1\circ f_2 \ot g_1\circ g_2).
\]

The monoidal $2$-categories $\big(\CV,\odot_{\rm S}\big)$ and
$\big(\CV,\odot_{\rm M}\big)$ are isomorphic as mo\-noidal categories, via the
isomorphism $\pi$ given by the identity on $\CV$ and 
the natural family of~functors
\[
\big\{
\varpi_{\ttA,\ttB} :   \ttA \odotS \ttB \stackrel \cong
\longrightarrow  
\ttA \odotM \ttB, \ \ttA,\ttB \in \CV\big\}.
\]
The functor $\varpi_{\ttA,\ttB}$ 
is the identity  on objects, while on
a morphism 
\[
f \ot g \in \ttA(a_1,a_2) \ot \ttB(b_1,b_2) =
(\ttA \odotS \ttB)\big((a_1,b_1),(a_2,b_2)\big)
\]
it acts as
\[
\varpi_{\ttA,\ttB}(f \ot g) : = (-1)^{|f||g|} (f \ot g)
 \in \ttA(a_1,a_2) \ot \ttB(b_1,b_2) = 
(\ttA \odotM \ttB)\big((a_1,b_1),(a_2,b_2)\big).
\]

The category  $\Lin$ of graded vector 
spaces and their homogeneous 
linear maps of arbitrary degrees is naturally enriched over $\Vect$.
It turns out that also $\Lin$   admits two $\Vect$-enriched symmetric monoidal
structures, the standard one and the
{McGill\/} monoidal structure defined~below.

The monoidal product of objects is for both structures
the usual tensor product of graded vector spaces, 
but the products  differ by their actions on
morphisms.  
The standard convention is that, for homogeneous maps $f:V' \to W'$,
$g: V'' \to W''$ and homogeneous vectors $u \in V'$, $v \in W'$ one
defines
\begin{subequations}
\begin{equation}
\label{eq:61}
(f \ot g) (u \ot v) = (-1)^{|g||u|} f(u) \ot g(v), 
\end{equation}
while some categorists at McGill University in Montreal would prefer
\begin{equation}
\label{eq:62}
(f \ot g) (u \ot v) = (-1)^{|f||v|} f(u)\ot g(v).
\end{equation}
\end{subequations}
The second convention would follow from the Koszul sign rule if we
are applying the morphisms from the right. Equation~(\ref{eq:62})
would then read as
\[ 
 (u \ot v)(f \ot g) = (-1)^{|f||v|} f(u)\ot g(v),
\] 
the unexpected sign coming from commuting $f$ over $v$.
We denote, only for the purposes of this section, 
the first monoidal structure by $\ot_{\rm S}$ and second by $\ot_{\rm M}$ (`S' 
abbreviating {\underline s}tandard and
`{M}' {\underline M}cGill). The corresponding 
monoidal categories will be denoted by $\Lin_{\rm S}$ and $\Lin_{\rm M}$,
respectively. Notice that both monoidal structures coincide on the
subcategory $\Vect$  of graded vector
spaces and their linear degree $0$ maps. 

It can be easily verified that $\Lin_{\rm S}$ is a symmetric
pseudomonoid in $\big(\CV,\odotS\big)$ 
while $\Lin_{\rm M}$   a symmetric
pseudomonoid in $\big(\CV,\odotM\big)$. The isomorphism
\[
\pi : \big(\CV,\odot_{\rm S}\big) \cong \big(\CV,\odot_{\rm M}\big)
\]
induces an isomorphism
\[
{\it PsMon}(\pi) : {\it PsMon}\big(\CV,\odot_{\rm S}\big) \cong 
 {\it PsMon}\big(\CV,\odot_{\rm M}\big)
\]
 of the corresponding
categories of pseudomonoids, and
\[
{\it PsMon}(\pi)(\Lin_{\rm S}) = \Lin_{\rm M}.
\]

We finish this section by a kindergarten example of an odd structure
in $\Lin$. 
An {\em odd associative algebra\/} is a couple $A = (A,\bullet)$
consisting of a graded vector space $A$ and a degree $+1$ operation 
\hbox{$\bullet : A \ot A \to A$} which is anti-associative, i.e.
\begin{equation}
\label{koupil_jsem_si_maleho_Titana}
\bullet(\id \ot \bullet) + \bullet(\bullet \ot \id)   = 0.
\end{equation}
Since 
the structure operation  has an odd
degree, the form of axiom~(\ref{koupil_jsem_si_maleho_Titana}) evaluated at
concrete elements may depend on the chosen monoidal structure of $\Lin$. 
Let us calculate
\[
\bullet (\id \ot \bullet) (x \ot y \ot z)  + \bullet  (\bullet \ot
\id)(x \ot y \ot z).
\]
for all $x,y,z \in A$.
While in $\Lin_{\rm S}$ we get
\begin{equation*}
(-1)^{|x|}x \bullet ( y \bullet  z) + (x \bullet  y) \bullet z,
\end{equation*}
in $\Lin_{\rm M}$ we obtain
\begin{equation*}
 x \bullet  ( y \bullet z)  + (-1)^{|z|}(x \bullet y) \bullet z.
\end{equation*}

The  categories of odd associative 
algebras in $\Lin_{\rm S}$ and in $\Lin_{\rm M}$ are however isomorphic. 
Indeed, the modification
\[
x \bullet y \longmapsto (-1)^{|x| + |y|} x \bullet y
\]
turns an odd associative algebra in $\Lin_{\rm S}$ into one in $\Lin_{\rm M}$ and vice
versa. In fact, there exists an one-to-one correspondence between odd
associative algebra structures on $A$ and usual associative algebra
structures on the
suspension $\susp A$. The corresponding associative product $\circ$ is
given by the commutative diagram   
\[
\unitlength=.9pt
\begin{picture}(110.00,60.00)(0.00,0.00)
\put(55.00,4.00){\makebox(0.00,0.00)[b]{\scriptsize $\bullet$}}
\put(104.00,25.00){\makebox(0.00,0.00)[l]{\scriptsize$\susp$}}
\put(6.00,25.00){\makebox(0.00,0.00)[r]{\scriptsize$\desusp \ot \desusp\ $}}
\put(57.00,54.00){\makebox(0.00,0.00)[b]{\scriptsize$\circ$}}
\put(100.00,40.00){\vector(0,-1){30.00}}
\put(10.00,40.00){\vector(0,-1){30.00}}
\put(35.00,0.00){\vector(1,0){50.00}}
\put(40.00,50.00){\vector(1,0){45.00}}
\put(100,0.00){\makebox(0.00,0.00){\hphantom{.}$A$.}}
\put(10.00,0.00){\makebox(0.00,0.00){$A \ot A$}}
\put(100.00,50.00){\makebox(0.00,0.00){$\susp A$}}
\put(10.00,50.00){\makebox(0.00,0.00){$\susp A \ \ot \!\susp A$}}
\end{picture}
\]
If not stated otherwise, we will use in the rest of this note the
standard monoidal structure and drop the subscript ${\rm S}$.

\section{Odd modular {\op}s}
\label{Mikes_mi_popral_k_narozkam}

Odd modular operads are particular cases of twisted modular
operads which  were introduced in the classical 1998 paper
\cite[(4.2)]{getzler-kapranov:CompM98}.\footnote{More precisely, 
they are ${\mathfrak K}$-twisted
  operads, where ${\mathfrak K}$ is the dualizing
  cocycle~\cite[(4.8)]{getzler-kapranov:CompM98}.} 
They were originally defined as
algebras for a~certain monad of decorated graphs. For practical 
calculations it is however convenient to have a
{\em biased\/} definition which appeared much later in
\cite{doubek:_modul}. For the convenience of the reader we repeat it below. 
The first definition in which $\Vect$
denotes the category of graded vector space and their degree-$0$
linear maps is however standard.

\begin{definition}
\label{dnes_prijel_Ivo_do_Srni_bis}
A {\em modular module\/} (in $\Vect$) is a covariant functor
\[
E : \Cor \times \bbA  \rightarrow \Vect
\]
from the cartesian product of the category $\Cor$ of finite sets and
their isomorphisms with the
discrete category of natural numbers to $\Vect$.
\end{definition}

Explicitly, a modular module $E$ is a collection $E(S;g)$, $S \in \Cor$, 
$g \in \bbA$, of graded vector spaces together with
functorial degree $0$ morphisms 
\[
E(\sigma) : E(S;g) \to E(T;g)
\]
specified
for any isomorphism\footnote{Isomorphisms are the only morphisms in
  $\Cor$ by definition.} $\sigma : S\ \redukce{$\stackrel \cong\lra$}\ T$
and $g \in \bbA$.

Roughly  speaking, an odd modular {\op} is a modular {\op} whose
structure operations have `wrong' degrees and some axioms acquire
`wrong' signs. Let us give a precise:

\begin{definition}
\label{modular}
An {\em odd modular {\op}\/} is  a modular module
\[
\oM = \big\{\oM(S;g) \in \Vect\  \vrt \  (S;g)   \in  \Cor \times \bbA \big\}
\]
together with degree $+1$ morphisms ($\twooo ab$-operations)
\begin{equation}
\label{v_Galway}
\twooo ab:\oM\big(S_1 \sqcup \stt a;g_1\big)
\otimes \oM\big(S_2\sqcup \stt b;g_2\big)  
\to \oM ( S_1\sqcup S_2;g_1+g_2)
\end{equation}
defined for arbitrary disjoint finite 
sets $S_1$, $S_2$, symbols $a,b$, and 
arbitrary $g_1,g_2 \in \bbA$.  There are, moreover, 
degree $1$ morphisms (the contractions)
\[
\twxxi_{uv} = \twxxi_{vu} : \oM\bl S \sqcup \stt {u,v} ;g\br  
\to \oM(S ;g+1)
\]
given for any finite
set $S$,  $g \in \bbA$, and symbols $u,v$.\footnote{We are using the 
notation for
  structure operations of odd modular operads introduced in
\cite{ten-vul-Kaufmann}.}
These data are required to satisfy the following axioms.
\begin{enumerate}
\itemindent -1em
\itemsep .3em 
\item  [(i)]
For arbitrary isomorphisms $\rho : S_1 \sqcup \stt a \to  T_1$
  and $\sigma :  S_2 \sqcup \stt b \to  T_2$ of finite  
 sets and  $g_1$, $g_2 \in \bbA$, one has the equality
\[
\oM\bl\rho|_{S_1}\sqcup\sigma|_{S_2}\br 
\twooo ab =
\twooo{\rho(a)}{\sigma(b)} \ \big(\oM\(\rho\)\ot\oM\(\sigma\)\big)
\]
of maps 
\[
\oM\bl S_1 \sqcup \stt a ;g_1\br \otimes \oM \bl S_2 \sqcup \stt b
;g_2\br 
\to
\oM\bl T_1\sqcup T_2
\setminus  \{\rho(a),\sigma(b)\};g_1+g_2\br.
\]

\item [(ii)] 
For an isomorphism $\rho : S \sqcup \stt {u,v} \to T
  $ of finite sets and  $g \in \bbA$, one has the
  equality
\[
\oM\bl\rho|_S\br \ \twxxi_{uv} = 
\twxxi_{\rho(u)\rho(v)}\oM\(\rho\)
\]
of maps $\oM \bl S \sqcup \stt {u,v};g \br \to \oM\bl T \setminus
\{\rho(u),\rho(v)\} ;g+1\br$.

\item [(iii)]
For $S_1$, $S_2$, $a$, $b$ and $g_1$, $g_2$ as in~(\ref{v_Galway}),
one has the equality
\[
\twooo{a}{b} =  \twooo{b}{a} \tau
\]
of maps $\oM\(S_1 \sqcup \stt a;g_1\)\otimes \oM\(S_2 \sqcup \stt b;g_2\)
\to \oM\bl S_1\sqcup S_2;g_1+g_2\br$.\footnote{Recall that $\tau$ is
  the commutativity constraint~(\ref{flip}) in the category of graded vector
  spaces.}

\item  [(iv)]
For mutually disjoint   sets
  $S_1,S_2,S_3$, symbols  
$a, b,c, d$ and $g_1,g_2,g_3 \in \bbA$,   one has the equality
\[
\label{Dnes_s_Jaruskou_k_Pakouskum}
\twooo ab (\id \ot \twooo cd)  = -\twooo cd (\twooo ab \ot
\id)
\]
of maps from
$\oM \bl S_1 \sqcup \stt a;g_1 \br \ot \oM\bl S_2  \sqcup \stt {b,c};g_2\br 
\ot  \oM\bl S_3 \sqcup \stt d;g_3\br$  to the space
$\oM\bl S_1 \sqcup S_2 \sqcup S_3;
g_1\!+\!
g_2\! +\! g_3\br$.

\item  [(v)]
For a finite  set $S$, symbols $a,b,c,d$ and  $g \in \bbA$
one has the equality
\[
\twxxi_{ab} \ \twxxi_{cd} =- \twxxi_{cd} \ \twxxi_{ab}
\]
of maps $\oM\bl S \sqcup \{a,b,c,d\};g\br 
\to \oM (S;g+2)$.

\item [(vi)] 
For finite sets $S_1, S_2$, symbols $a,b,c,d$ and  $g_1,g_2
  \in \bbA$, one has the equality
\[
\twxxi_{ab} \ \twooo{c}{d} = -\twxxi_{cd} \ \twooo{a}{b}
\]
of maps $\oM\bl S_1 
\sqcup \stt {a,c};g_1\br \ot \oM \bl S_2 \sqcup \stt {b,d};g_2\br
\to \oM( S_1 \sqcup S_2;g_1+g_2+1)$.

\item [(vii)] 
For  finite sets $S_1, S_2$, symbols
$a,b,u,v$, and $g_1,g_2 \in \bbA$, one has the equality
\[
\twooo{a}{b} \ (\twxxi_{uv}\ot\id) = -\twxxi_{uv} \ \twooo{a}{b}
\]
of maps $\oM\bl
S_1 \sqcup \stt{a,u,v};g_1\br \ot \oM \bl S_2 \sqcup 
\stt b;g_2\br \to \oM (S_1 \sqcup S_2 ;g_1+g_2+1)$.
\end{enumerate} 
\end{definition}

As for odd associative algebras discussed in Section~\ref{zase_prsi},
the form of some axioms  of odd modular {\op}s evaluated at
concrete elements depends on the chosen monoidal structure of $\Lin$.  
Let us, for instance, evaluate axiom~(iv) at homogeneous elements 
\[
x \in \oM \bl S_1 \sqcup \stt a;g_1 \br,\
y \in  \oM\bl S_2  \sqcup \stt {b,c};g_2\br 
\ \hbox { and } \ z \in 
\oM\bl S_3 \sqcup \stt d;g_3\br
\]
i.e.~expand
\[
\twooo ab (\id \ot \twooo cd)(x \ot y \ot z)  = - \twooo cd (\twooo ab \ot
\id)(x \ot y \ot z).
\]
While in $\Lin_{\rm S}$ we get
\begin{subequations}
\begin{equation}
\label{dnes_Etiopska_restaurace}
(-1)^{|x|}x \twooo ab ( y \twooo cd z) = - (x \twooo ab y) \twooo cd z,
\end{equation}
in $\Lin_{\rm M}$ we obtain
\begin{equation}
\label{dnes_Etiopska_restaurace_1}
 x \twooo ab ( y \twooo cd z) = - (-1)^{|z|}(x \twooo ab y) \twooo cd z.
\end{equation}
\end{subequations}
\begin{subequations}
Likewise, axiom~(vii) in $\Lin_{\rm S}$ reads
\begin{equation}
\label{za_tyden_do_Prahy}
\twxxi_{uv}(x) \twooo{a}{b} y
 = -\ \twxxi_{uv} (x \twooo{a}{b} y)
\end{equation}
while in $\Lin_{\rm M}$ one would get
\begin{equation}
\label{za_tyden_do_Prahy_s_Germanwings}
(-1)^{|y|}\twxxi_{uv}(x) \twooo{a}{b} y
 = -\ \twxxi_{uv} (x \twooo{a}{b} y)
\end{equation}
\end{subequations}
for $x,y$ belonging to the appropriate components of $\oM$.
The remaining axioms are the same in both
monoidal structures.

It turns our that the  categories of odd modular
{\op}s in $\Lin_{\rm S}$ and in $\Lin_{\rm M}$ are isomorphic;
the modification
\begin{equation}
\label{posledni_den_konference}
x \twooo  ab y \mapsto (-1)^{|x| + |y|} x \twooo ab y,
\   \twxxi_{uv}(x) \mapsto (-1)^{|x|} \twxxi_{uv}(x),
\end{equation}
turns an odd modular {\op} in $\Lin_{\rm S}$ into one in $\Lin_{\rm M}$ and vice
versa.

It is however not true that an odd modular {\op} structure is the same as an
ordinary one on the suspension of the underlying modular module. While
the suspended $\twooo ab$-operations are of degree $0$ as for the
ordinary modular {\op}s, the suspended contractions $\twxxi_{uv}$
retain degree~$1$. The categories of ordinary and odd modular {\op}s
are genuinely different.

\section{Odd endomorphism {\op}s}
\label{zitra_letim_do_Prahy}

The classical definition of the 
(ordinary) modular endomorphism {\op} $\End_V$ given
in \hbox{\cite[(1.7)]{getzler-kapranov:CompM98}}  requires as 
the input data a graded
vector space $V$ with a non-degenerate symmetric bilinear form $B : V \ot V \to
\bfk$ of degree $0$. For $[n] := 
\{1,\ldots,n\}$  one  puts
\[
\End_V\big([n]\big) := V^{\ot n}, \ n \geq 0,
\] 
with the operadic structure given by `contracting indexes'
using $B$.
If $B$ has degree $+1$, the same construction leads to an odd
modular endomorphism operad  \cite[Example 5.3]{markl:la}.   

One easily observes that $B$ need not be non-degenerate
-- everything makes sense even in the extreme case when $B =
0$. Moreover, in mathematical physics, it is more natural to work in
the dual settig with a symmetric degree $+1$ tensor \hbox{$s \in V\!
  \ot\! V$} instead
of~$B$, and the components of the odd endomorphism operad given by 
\[
\End_V\big([n]\big) := \big(V^{\ot n}\big)^*, \ n \geq 0.
\]
The operadic structure is given by `expanding indexes'
using $s$. We will focus to this version of the odd endomorphism operad.

It turns out that there are two interpretations
what expanding indexes means. The   
first, seemingly preferable one, is expressed by~(\ref{psani_v_Bonnu})
and~(\ref{pujdeme_dnes_k_Mechacum?}) below. It does not involve
duals and uses only canonical isomorphisms. The other one, represented
by~(\ref{eq:prvni_psani_v_Bonnu}) and~(\ref{66}), is much less aesthetically
pleasing since it needs duals and inclusions of the form
\[
A^* \ot B^* \hookrightarrow (A\ot
B)^*
\]
which from seemingly random reasons go in the desired direction. It is
the peculiarity of the odd case that both constructions lead to
different results. If we assume the standard monoidal structure of
$\Lin$, then the correct result is given by the second, ugly one.

We will need the tensor
product of a family  $\{V_c\}_{c \in S}$ of graded vector spaces
indexed by a~finite set $S \in \Cor$. 
Since $S$ is not a priory ordered, we want a concept 
that would not depend on a chosen order. The idea is to choose an
order, then perform the usual tensor product, and then identify
the products over different orders using the Koszul sign rule.
Since an order of a finite set $S$ with $n$ elements is the
same as an isomorphism $\omega : \stt{1,\ldots,n} \redukce{{
    $\stackrel\cong\to$ }}
S$, we are led to the following:

\begin{definition}
\label{nakonec_jsem_nejel_protoze_prselo}
The {\em unordered tensor product\/} $\bigotimes_{c \in S} V_c$ of the
collection  $\{V_c\}_{c \in S}$ is
the vector space of equivalence classes of usual tensor products
\begin{equation}
\label{v_patek_domu_za_Jaruskou}
v_{\omega(1)} \ot \cdots \ot v_{\omega(n)} \in V_{\omega(1)} 
\ot \cdots \ot V_{\omega(n)},\ \omega : \stt{1,\ldots,n} 
\redukce{ $\stackrel\cong\lra$ } S,
\end{equation}
modulo the identifications
\[
v_{\omega(1)} \ot \cdots \ot v_{\omega(n)} 
\sim \epsilon(\sigma)\
v_{\omega\sigma(1)} \ot \cdots \ot v_{\omega\sigma(n)},\ \sigma \in \Sigma_n,
\]
where $\epsilon(\sigma)$ is the Koszul sign~(\ref{Koszul_sign}) of the
permutation $\sigma$.
\end{definition}

The need for a subtler version of the tensor product is caused by
  the fact that the category $\Vect$ of graded vector spaces is a
 symmetric monoidal category with a {\em non-trivial symmetry\/}. 
Similar unordered products can
  be defined in any symmetric monoidal category with finite colimits,
see e.g.~\cite[Def.~II.1.58]{markl-shnider-stasheff:book}.
Let us formulate two important properties of unordered tensor
products. 

\begin{lemma}
\label{bila_nemoc}
Let
$\sigma : S \to D$ be an isomorphism of finite sets, $\stt {V_c}_{c\in S}$
and $\stt {W_d}_{d\in D}$ collections of graded vector spaces, and
$\varphi = \stt{\varphi_c : V_c \to W_{\sigma c}}_{c \in S}$ a family of
linear maps. 
Then the assignment
\[
\bigotimes_{c\in S}V_c \ni \big[v_{\omega(1)}\ot \cdots \ot v_{\omega(n)}\big]
\longmapsto \big[w_{\sigma\omega(1)}\ot \cdots \ot
w_{\sigma\omega(n)}\big]
\in \bigotimes_{d \in D}W_d
\]
with
$w_{\sigma\omega(i)} := \varphi_{\omega(i)}(v_{\omega(i)}) \in
W_{\sigma\omega(i)}$, $1 \leq i
\leq n$,
defines a natural map
\[
\overline{(\sigma,\varphi)} :
\bigotimes_{c \in S}V_c \to \bigotimes_{d \in D}W_d
\]
of unordered products
\end{lemma}

\begin{proof}
A direct verification.
\end{proof}

A particularly important case of the above lemma is when $V_c = V_d =
V$ for all $c \in S$, $d \in D$, and $\varphi_c : V \to V$ is the
identity for all $c\in S$. Lemma~\ref{bila_nemoc} then gives a~natural
map
\begin{equation}
\label{jdu_si_zabehat}
\overline \sigma : =\overline{(\sigma,\varphi)} :
\bigotimes_{c \in S}V_c \to \bigotimes_{d \in D}V_d.
\end{equation}

\begin{lemma}
\label{Tequila_v_lednici}
For disjoint finite sets $S', S''$, one has a canonical 
isomorphism
\[
\bigotimes_{c' \in S'} V_{c'} \ot \bigotimes_{c'' \in S''}V_{c''} 
\cong
\bigotimes_{c\in S'\sqcup\, S''} V_c.
\]
\end{lemma}

\begin{proof}
Each $\omega' : \stt{1,\ldots,n} \redukce{ $ \stackrel\cong\to $ }
S'$ and $\omega'' : \stt{1,\ldots,m} \redukce{ $ \stackrel\cong\to $ }
S''$ determine an isomorphism
\[
\omega' \sqcup \omega'' : \stt{\rada 1{n+m}}  \redukce{ $
  \stackrel\cong\lra $ }
 S'\sqcup S''
\]
by the formula
\[
(\omega' \sqcup \omega'')(i) :=
\begin{cases}
\omega'(i),&\hbox{if $1 \leq i \leq n$, and}
\cr
\omega''(i-n),&\hbox{if $n < i \leq n+m$\ .}
\end{cases}
\]
The isomorphism of the lemma is then given by the assignment
\[
[v_{\omega'(1)} \ot \cdots \ot v_{\omega'(n)}] \ot
[v_{\omega''(1)} \ot \cdots \ot v_{\omega''(m)}]
\longmapsto [v_{(\omega' \sqcup\, \omega'')(1)} \ot \cdots \ot 
v_{(\omega' \sqcup\, \omega'')(n+m)}].
\]
This finishes the proof.
\end{proof}

The input data of the odd modular endomorphism
{\op} is a graded vector space $V$ with 
a symmetric tensor $s \in V\! \ot\! V$ of degree $+1$. The symmetry
means  that 
$\tau(s) = s$, where $\tau$ is the interchange~(\ref{flip}). We
will interpret as usual 
$s$ as a linear degree $+1$ map $s: \bfk \to V\ot V$. 
For a finite set 
$S$~put
\[\textstyle
\End_V(S) := \Lin\big(\bigotimes_{c \in S} V_c,\bfk\big) = 
(\bigotimes_{c \in S} V_c)^\dual
\]
where $V_c := V$ for each $c
\in S$. 
Given an isomorphism $\sigma : S \to D$ of finite sets, we define the
induced map
\[
\End_V(\sigma)  : \End_V(S) \to \End_V(D)
\]
by $\End_V(\sigma)(x) := x \, \overline \sigma^{-1}$  
for $x :\bigotimes_{c \in S} V_c \to \bfk \in
\End_V(S)$ and $\overline\sigma$ as in~(\ref{jdu_si_zabehat}).

Let $S_1,S_2$ be disjoint finite sets and $a \not= b$ two symbols. Our
next task will be to define, for linear functionals
\[
x \in \End_V\big(S_1 \sqcup \stt a\big)
\ \hbox { and } \ y \in \End_V\big(S_2
\sqcup \stt b\big),
\]
their $\twooo ab$-product $x\twooo ab y \in \End_V\big(S_1 \sqcup
S_2\big)$. The  natural choice is obviously the composition
\begin{align*}
\bigotimes_{c \in S_1 \sqcup S_2}& \hskip -.5em V_{c} \stackrel\cong\lra
 \bigotimes_{c' \in S_1} V_{c'} \ot
\bigotimes_{c'' \in S_2} V_{c''}
 \stackrel\cong\lra
 \bigotimes_{c' \in S_1} V_{c'} \ot \bfk \ot
\bigotimes_{c'' \in S_2} V_{c''}
\\ 
\nonumber 
& \stackrel{\sid \ot s \ot \sid}\vlra   \bigotimes_{c' \in S_1} V_{c'}
\ot V_a \ot V_b \ot
\bigotimes_{c'' \in S_2} V_{c''} 
\stackrel\cong\lra\bigotimes_{c' \in S_1 \sqcup \stt a}\hskip -.5em V_{c'}
\ot \bigotimes_{c'' \in S_2 \sqcup \stt b}\hskip -.5em
 V_{c''} \stackrel{x \ot y}\vlra \bfk
\end{align*}
in which the isomorphisms are those of
Lemma~\ref{Tequila_v_lednici}.   In shorthand,
\[
x\twooo ab y = (x \ot y)\big(\id_{V^{\ot S_1}} \ot s \ot \id_{V^{\ot
    S_2}}\big)
\]
and, denoting $\id_{S_1} := \id_{V^{\ot S_1}}$ and  
$\id_{S_2} := \id_{V^{\ot S_2}}$,  still more concisely
\begin{subequations}
\begin{equation}
\label{psani_v_Bonnu}
x\twooo ab y := (x \ot y) \big(\id_{S_1} \ot s \ot \id_{S_2}\big).
\end{equation}
Alternatively, one may define $x\twooo ab y$ as the result of the application
of the composition
\begin{align*}
\Big(\bigotimes_{c' \in S_1 \sqcup \stt a}&\hskip -.5em V_{c'}\Big)^\dual
\ot 
\Big(\bigotimes_{c'' \in S_2 \sqcup \stt b}\hskip -.5em
 V_{c''}\Big) ^\dual
\hookrightarrow 
\Big(\bigotimes_{c' \in S_1 \sqcup \stt a}\hskip -.5em V_{c'}
\ot
\bigotimes_{c'' \in S_2 \sqcup \stt b}\hskip -.5em
 V_{c''}\Big) ^\dual
\\
\nonumber 
\stackrel\cong\lra &
\Big(
\bigotimes_{c' \in S_1} V_{c'}
\ot V_a\ot V_b \ot
\bigotimes_{c'' \in S_2} V_{c''} \Big)^\dual
 \stackrel{(\sid \ot s \ot \sid)^\dual}\vlra
\Big(\bigotimes_{c' \in S_1} V_{c'} \ot \bfk \ot
\bigotimes_{c'' \in S_2} V_{c''} \Big)^\dual
\\
\nonumber 
\stackrel\cong\lra &
\Big( \bigotimes_{c' \in S_1} V_{c'} \ot
\bigotimes_{c'' \in S_2} V_{c''} \Big)^\dual
\stackrel\cong\lra
\Big(\bigotimes_{c \in S_1 \sqcup S_2} \hskip -.5em V_{c}\Big)^\dual
\end{align*}
to $x \ot y \in \big(\bigotimes_{c' \in S_1 \sqcup \stt a} V_{c'}\big)^\dual
\ot \big(\bigotimes_{c'' \in S_2 \sqcup \stt b} 
V_{c''}\big) ^\dual$. In shorthand,
\begin{equation}
\label{eq:prvni_psani_v_Bonnu}
x\twooo ab y := \big(\id_{S_1} \ot s \ot \id_{S_2}\big)^\dual(x \ot y) .
\end{equation}
\end{subequations}
We shall keep in mind 
that~(\ref{eq:prvni_psani_v_Bonnu}) implicitly involves 
canonical identifications and inclusions. 

An obvious way to define the contraction
 $\twxxi_{uv}x \in \End_V(S;g+1)$ of a linear functional $x \in \End_V\big(S
\sqcup \{u,v\};g\big)$ is the composition
\begin{equation*}
\bigotimes_{c \in S} V_{c} \cong \bfk \ot \bigotimes_{c \in S} V_{c}
\stackrel{s \ot \sid}\vlra  
V_u \ot V_v \ot
\bigotimes_{c \in S} V_{c} 
\stackrel\cong\lra\bigotimes_{c \in S \sqcup \stt {u,v}} V_{c}
 \stackrel{x}\lra \bfk.
\end{equation*}
In shorthand,
\begin{subequations}
\begin{equation}
\label{pujdeme_dnes_k_Mechacum?}
\twxxi_{uv}x :=  x \big(s \ot \id_S).
\end{equation}
The commutativity of the diagram 
\begin{equation*}
\raisebox{-5em}{}
{
\thicklines
\unitlength=1.2pt
\begin{picture}(180.00,50.00)(0.00,50.00)
\put(150.00,20.00){\vector(1,1){20.00}}
\put(70.00,10.00){\vector(1,0){27.00}}
\put(10.00,40.00){\vector(1,-1){20.00}}
\put(150.00,80.00){\vector(1,-1){20.00}}
\put(70.00,90.00){\vector(1,0){30.00}}
\put(10.00,60.00){\vector(1,1){20.00}}
\put(85.00,3.00){\makebox(0.00,0.00){\scriptsize $\id \ot s$}}
\put(165.00,25.00){\makebox(0.00,0.00){\scriptsize $\cong$}}
\put(165.00,75.00){\makebox(0.00,0.00){\scriptsize $\cong$}}
\put(85.00,97.00){\makebox(0.00,0.00){\scriptsize $s \ot \id$}}
\put(15.00,75.00){\makebox(0.00,0.00){\scriptsize $\cong$}}
\put(15.00,25.00){\makebox(0.00,0.00){\scriptsize $\cong$}}
\put(40.00,10.00){\makebox(0.00,0.00){$\bigotimes_{c \in S} V_{c} \ot \bfk$}}
\put(140.00,10.00){\makebox(0.00,0.00){$
\bigotimes_{c \in S} V_{c} \ot V_u \ot V_v$}}
\put(180.00,50.00){\makebox(0.00,0.00){$\bigotimes_{c \in S \sqcup
      \stt {u,v}} V_{c}$\ .}}
\put(140.00,90.00){\makebox(0.00,0.00){$V_u \ot V_v \ot
\bigotimes_{c \in S} V_{c}$}}
\put(40.00,90.00){\makebox(0.00,0.00){$\bfk \ot \bigotimes_{c \in S} V_{c}$}}
\put(0.00,50.00){\makebox(0.00,0.00){$\bigotimes_{c \in S} V_{c}$}}
\end{picture}}
\end{equation*}
implies that we could replace $(s\ot \id_S)$ 
in~(\ref{pujdeme_dnes_k_Mechacum?}) 
by $(\id_S \ot s)$ with the same result. We could in fact place $s$
into an arbitrary position without affecting the result. A similar
remark applies also to~(\ref{psani_v_Bonnu})
and~(\ref{eq:prvni_psani_v_Bonnu}).
 
Another possibility is to define
$\twxxi_{uv}x$ as the result of the application
\begin{equation*}
\Big(\bigotimes_{c \in S \sqcup \stt {u,v}} V_{c}\Big)^\dual
\stackrel\cong\lra
\Big(V_u \ot V_v \ot
\bigotimes_{c \in S} V_{c}\Big) ^\dual
\stackrel{(s \ot \sid)^\dual}\vlra 
\Big(\bfk \ot \bigotimes_{c \in S} V_{c}\Big)^\dual \cong 
\Big(\bigotimes_{c \in S} V_{c}\Big)^\dual
\end{equation*}
to $x \in\big(\bigotimes_{c \in S \sqcup \stt{u,v}} V_c\big)^\dual$. 
In shorthand, 
\begin{equation}
\label{66}
\twxxi_{uv}x := \big(s \ot \id_S)^\dual (x).
\end{equation}
\end{subequations}

Here comes a surprise. Since $|s|=1$, the two definitions
of the $\twooo ab$-operation, i.e.~the one 
via~(\ref{psani_v_Bonnu}) and the one 
via~(\ref{eq:prvni_psani_v_Bonnu}), lead to
different results! The reason is that they are not dual to each other,
since the duality~(\ref{dostane_Jarka_dovcu?}) 
acquires a nontrivial sign. The resulting
$x\twooo ab y$'s differ by $(-1)^{|x|+|y|}$.

Likewise, definitions~(\ref{pujdeme_dnes_k_Mechacum?})
and~(\ref{66}) are not dual to each other and
the resulting $\twxxi_{uv}(x)$'s differ by  $(-1)^{|x|}$.
What happens is described in the following proposition;
recall that $\Lin_{\rm S}$ and $\Lin_{\rm M}$ denote the two versions
of the category $\Lin$ discussed in Section~\ref{zase_prsi}.

\begin{proposition}
\label{pokusim_se_vratit_kufr}
The modular collection $\End_V$ with operations $\twooo ab$ and
$\twxxi_{uv}$ defined by~(\ref{eq:prvni_psani_v_Bonnu})
and~(\ref{66}) is an odd modular {\op} in
$\Lin_{\rm S}$ while~(\ref{psani_v_Bonnu}) 
and~(\ref{pujdeme_dnes_k_Mechacum?}) give an odd modular {\op} in
$\Lin_{\rm M}$. 
\end{proposition}

\begin{proof}
Let us show that the $\twooo ab$-operations defined 
by~(\ref{eq:prvni_psani_v_Bonnu})
satisfy~(\ref{dnes_Etiopska_restaurace}). In the following
calculations, $\bar s$ and $\bar{\bar s}$ are two copies of the map
$s : \bfk \to V \ot V$. One has
\begin{align*}
x\twooo ab (y \twooo cd z)& =
\big(\id_{S_1} \ot \bar s \ot \id_{S_2 \sqcup S_3}\big)^\dual\big(x \ot (y
                        \twooo cd z)\big)
\\
&= \big(\id_{S_1} \ot \bar s \ot \id_{S_2 \sqcup S_3}\big)^\dual\big(x
  \ot 
(\id_{S_2\sqcup \stt b} \ot \bar{\bar s} \ot \id_{S_3})^\dual
(y  \ot z)\big)
\\
&=
(-1)^{|x|}\big(\id_{S_1} \ot \bar s \ot \id_{S_2 \sqcup S_3}\big)^\dual
\big(\id_{S_1 \sqcup \{a,b\} \sqcup S_2} 
\ot \bar{\bar s} \ot \id_{S_3}\big)^\dual(x \ot y \ot z)
\end{align*}
while
\begin{align*}
\label{ne_citaet}
(x\twooo ab y) \twooo cd z 
&=
\big(\id_{S_1 \sqcup S_2} \ot \bar{\bar s} \ot \id_{S_3}\big)^\dual
\big((x \twooo ab y) \ot z\big)
\\
&=
\big(\id_{S_1 \sqcup S_2} \ot \bar{\bar s} \ot \id_{S_3}\big)^\dual
\big((\id_{S_1}\ot \bar s\ot \id_{S_2 \sqcup \stt c})^\dual(x \ot y) \ot z\big)
\\
&=
\big(\id_{S_1 \sqcup S_2} \ot \bar{\bar s} \ot \id_{S_3}\big)^\dual
\big(\id_{S_1} 
\ot \bar s \ot \id_{S_2  \sqcup \{c,d\} \sqcup S_3}\big)^\dual(x \ot y \ot z).
\end{align*}
To finish the proof of~(\ref{dnes_Etiopska_restaurace}), we observe that
\begin{align*}
\big(\id_{S_1} \ot {\bar s} \ot \id_{S_2} \ot \bar{\bar s} \ot
\id_{S_3}\big)^\dual  &=\big(\id_{S_1} \ot 
\bar s \ot \id_{S_2 \sqcup S_3}\big)^\dual
\big(\id_{S_1 \sqcup \{a,b\} \sqcup S_2} 
\ot \bar{\bar s} \ot \id_{S_3}\big)^\dual
\\
&=-\big(\id_{S_1 \sqcup S_2} \ot \bar{\bar s} \ot \id_{S_3}\big)^\dual
\big(\id_{S_1} 
\ot \bar s \ot \id_{S_2  \sqcup \{c,d\} \sqcup S_3}\big)^\dual,
\end{align*}
the minus sign coming from commuting $\bar s$ over   $ \bar{\bar s}$.

Let us also verify explicitly that the $\twooo ab$-operations defined 
by~(\ref{psani_v_Bonnu})
satisfy~(\ref{dnes_Etiopska_restaurace_1}). The related
calculation is of course obtained from the above one by removing duals and
inverting the order of compositions but, very crucially, 
{\em without\/} inserting Koszul signs. We obtain
\begin{align*}
x\twooo ab (y \twooo cd z)& =\big(x \ot (y
                        \twooo cd z)\big)
\big(\id_{S_1} \ot \bar s \ot \id_{S_2 \sqcup S_3}\big)
\\
&= \big(x
  \ot 
(y  \ot z)(\id_{S_2\sqcup \stt b} \ot \bar{\bar s} \ot \id_{S_3})
\big)
\big(\id_{S_1} \ot \bar s \ot \id_{S_2 \sqcup S_3}\big)
\\
&=(x \ot y \ot z)\big(\id_{S_1 \sqcup \{a,b\} \sqcup S_2} 
\ot \bar{\bar s} \ot \id_{S_3}\big)
\big(\id_{S_1} \ot \bar s \ot \id_{S_2 \sqcup S_3}\big)
\end{align*}
on one hand and
\begin{align*}
(x\twooo ab y) \twooo cd z 
&=
\big((x \twooo ab y) \ot z\big)
\big(\id_{S_1 \sqcup S_2} \ot \bar{\bar s} \ot \id_{S_3}\big)
\\
&=
\big((x \ot y)(\id_{S_1}\ot \bar s\ot \id_{S_2 \sqcup \stt c}) \ot
z\big)
\big(\id_{S_1 \sqcup S_2} \ot \bar{\bar s} \ot \id_{S_3}\big)
\\
&=(-1)^{|z|}(x \ot y \ot z)\big(\id_{S_1} 
\ot \bar s \ot \id_{S_2  \sqcup \{c,d\} \sqcup S_3}\big)
\big(\id_{S_1 \sqcup S_2} \ot \bar{\bar s} \ot \id_{S_3}\big)
\end{align*}
on the other.
Axiom~(\ref{dnes_Etiopska_restaurace_1}) now follows from the equality
\begin{align*}
\big(\id_{S_1} \ot {\bar s} \ot \id_{S_2} \ot \bar{\bar s} \ot
\id_{S_3}\big)
&=\big(\id_{S_1} 
\ot \bar s \ot \id_{S_2  \sqcup \{c,d\} \sqcup S_3}\big)
\big(\id_{S_1 \sqcup S_2} \ot \bar{\bar s} \ot \id_{S_3}\big)
\\
&=-\big(\id_{S_1 \sqcup \{a,b\} \sqcup S_2} 
\ot \bar{\bar s} \ot \id_{S_3}\big)\big(\id_{S_1} \ot 
\bar s \ot \id_{S_2 \sqcup S_3}\big).
\end{align*}

Notice that the sign difference between the results of the above two
computations is $(-1)^{|x|}$ versus $(-1)^{|z|}$ as it should be. The
verification of axioms~(\ref{za_tyden_do_Prahy})
resp.~(\ref{za_tyden_do_Prahy_s_Germanwings}) is similar. The
remaining axioms are not affected by the choice of the monoidal
structure in $\Lin$ so we will not verify them here.
\end{proof}

We evaluated~(\ref{psani_v_Bonnu}), (\ref{eq:prvni_psani_v_Bonnu}), 
(\ref{pujdeme_dnes_k_Mechacum?}) and~(\ref{66}) inside the standard
monoidal structure of $\Lin$. If we use the McGill one,
then~(\ref{psani_v_Bonnu})  
and~(\ref{pujdeme_dnes_k_Mechacum?}) would give an odd modular {\op} in
$\Lin_{\rm S}$ while~(\ref{eq:prvni_psani_v_Bonnu})
and~(\ref{66}) would lead to an odd modular {\op} in
$\Lin_{\rm M}$.

\def\cprime{$'$}\def\cprime{$'$}

\end{document}